\documentclass[12pt]{amsart}
\usepackage{amsmath}
\usepackage[a4paper, margin=1in]{geometry}
\usepackage{amsmath,amssymb}
\usepackage{float} 
\usepackage{graphicx,tikz}
\usepackage{booktabs} 
\usepackage{caption,amsmath} 
\usepackage[T1]{fontenc} 
\usepackage[utf8]{inputenc}
\usepackage{lmodern}
\usepackage{hyperref}
\usepackage[dvipsnames]{xcolor}
\newtheorem{proposition}{Proposition}[section]
\newtheorem{construction}{Construction}[section]

\newtheorem{theorem}{Theorem}[section]
\newtheorem{lemma}{Lemma}[section]
\newtheorem{corollary}{Corollary}[section]
\begin{document}
\title{\vspace{-1.5cm} 
   Polyominoes with maximal number of deep holes  
}

\author{Đorđe Baralić
    \and Shiven Uppal}

\address{ \scriptsize{Mathematical Institute SASA, Belgrade, Serbia }}
\email{djbaralic@mi.sanu.ac.rs}

\address{\scriptsize{Sanskriti School, New Delhi, India}}
\email{shivenuppal21@gmail.com}
    
\date{} 

\begin{abstract} 
In this paper, we study the extremal behaviour of deep holes in polyominoes. We determine the maximum number, $h_n$ of deep holes that an $n$-omino can enclose, ensuring that the boundary of each hole is disjoint from the boundaries of any other hole and from the outer boundary of the $n$-tile. Using the versatile application of Pick's theorem, we establish the lower and the upper bound for $h_n$, and show  that $h_n=\frac{n}{3}+o(n)$ asymptotically. To further develop these results, we compute $h_n$ as a function of $n$ for an infinite subset of positive integers. 
\end{abstract}

\maketitle 
\section{introduction}

The term \textit{'polyominoes'} was first introduced in \cite{Golomb} by Solomon W. Golomb in 1953 to describe a plane geometric figure that is formed by connecting a finite number of unit squares together along their edges. A polyomino is a subset of unit squares that has a connected interior. These unit squares are also referred to as \textit{cells}. Polyominoes were further popularized by Martin Gardner in `Mathematical Games' columns for \textit{Scientific American} magazine and by David Klarner in his research papers \cite{Klarner} and \cite{Klarner1}. They are a part of recreational mathematics and form the basis of many geometric combinatorial games such as tiling problems. A classic example is the tetromino, used in the game \textit{Tetris}, which consists of four cells connected in different combinations, \ref{tetro}. 

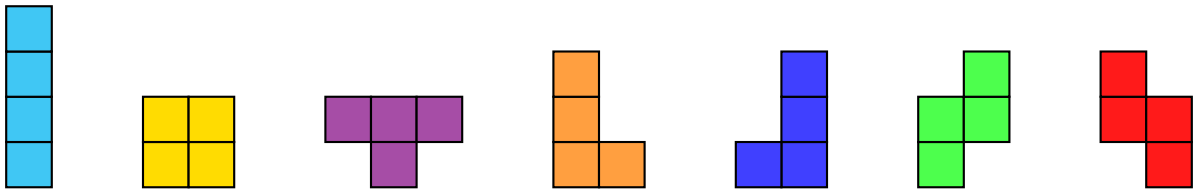
\begin{figure}[htbp]
\begin{center}
\begin{tikzpicture}[scale=0.6, every node/.style={font=\small}]

\foreach \y in {0,1,2,3} {
    \filldraw[thick, fill=cyan!60] (0,\y) rectangle ++(1,1);
}

\foreach \x in {3,4}
  \foreach \y in {0,1}
    \filldraw[thick, fill=yellow!80!orange] (\x,\y) rectangle ++(1,1);

\foreach \x in {7,8,9} {
    \filldraw[thick, fill=violet!70] (\x,1) rectangle ++(1,1);
}
\filldraw[thick, fill=violet!70] (8,0) rectangle ++(1,1);

\foreach \y in {0,1,2} {
    \filldraw[thick, fill=orange!75] (12,\y) rectangle ++(1,1);
}
\filldraw[thick, fill=orange!75] (13,0) rectangle ++(1,1);

\foreach \y in {0,1,2} {
    \filldraw[thick, fill=blue!75] (17,\y) rectangle ++(1,1);
}
\filldraw[thick, fill=blue!75] (16,0) rectangle ++(1,1);

\filldraw[thick, fill=green!70] (20,0) rectangle ++(1,1);
\filldraw[thick, fill=green!70] (20,1) rectangle ++(1,1);
\filldraw[thick, fill=green!70] (21,1) rectangle ++(1,1);
\filldraw[thick, fill=green!70] (21,2) rectangle ++(1,1);

\filldraw[thick, fill=red!90] (24,1) rectangle ++(1,1);
\filldraw[thick, fill=red!90] (24,2) rectangle ++(1,1);
\filldraw[thick, fill=red!90] (25,0) rectangle ++(1,1);
\filldraw[thick, fill=red!90] (25,1) rectangle ++(1,1);

\end{tikzpicture}
\caption{Tetris shapes}\label{tetro}
\end{center}
\end{figure}
\par

Polyominoes are named based on the area they enclose; a polyomino consisting of $n$ unit squares is called an $n$-omino. A polyomino is said to be \textit{without holes} if its interior is simply connected and its boundary forms a simple, unbroken loop. Problems related to simple polyominoes (polyominoes without holes) have been extensively studied in the literature. Harary and Harborth proved in \cite{harary} that the minimal possible perimeter of $n$-omino is $2 \lceil 2\sqrt{n}\rceil$. In contrast, in this paper, we aim to maximize the number of holes in an $n$-omino using Pick’s theorem. 
\par
Early groundwork by Eden Murray \cite{Murray} laid the foundation for understanding polyomino enumeration and growth, providing essential insight into the asymptotic combinatorial properties of polyominoes. He was the first to deduce that the number of distinct simply connected $n$-ominoes, $a_n$, satisfies the following inequality 
\begin{equation*}
(3.14)^n \leq a_n \leq \left (\frac{27}{4} \right)^n
\end{equation*}
\par
 In 1960s, Klarner's results \cite{Klarner} and \cite{Klarner1} further improved the lower bound for $a_n$ to
\begin{equation*}
    (3.73)^n \leq a_n
\end{equation*}
\par
He also proved that there exists $\lambda$ such that
\begin{equation*}
    \lim_{n \to \infty} \left(a(n)\right)^{1/n} = \lambda 
\end{equation*}
The limit for the growth constant of polyominoes $\lambda$ is called \textit{Klarner's constant}. The lower bound and upper bound limits were later improved to $3.980137 \leq \lambda$ (see \cite{Barequet}) and $\lambda \leq 4.5685$ (see \cite{Barequet1}), respectively.

\par 
The number of fixed $n$-ominoes $a_n$ is given by 
\begin{equation*}
    A(n) \sim C \lambda^n n^{\theta} 
\end{equation*}
for some constants $C > 0$ and $\theta \approx -1$. The conjecture is that $\lambda= 4.062569$.
\par 
Extending this concept, in her thesis \cite{Roldan}, \'{E}rika Rold\'{a}n-Roa studied the maximum number of holes that an $n$-omino can enclose, where a hole in a polyomino is defined as a bounded, connected component of the complement of the polyomino in a plane. In her notation, for a polyomino $A$, $h(A)$ denotes the number of holes in $A$, and $\mathcal{A}_n$ denotes the set of all $n$-ominoes. For $n \geq 1$, the function is 
\begin{equation*}
    f(n) = \max_{A \in \mathcal{A}_n} h(A)
\end{equation*}
For $h \geq 1$, the function $g(h)$ is defined as the minimum number of unit squares required to construct a polyomino with exactly $h$ holes.
\begin{equation*}
    g(h) := \min \{\, |A| \mid h(A) = h \,\}.
\end{equation*}
where $|A|$ denotes the number of squares.
\par
While polyominoes with fewer than seven tiles do not have a hole, polyominoes with holes grow exponentially faster than simply connected polyominoes. Kahle and Rold\'{a}n-Roa proved a tight bound for the asymptotic behaviour of the maximum number of holes in an $n$-omino as $n$ tends to infinity. Their result follows from the following inequality proved in \cite{Kahle}.
\begin{equation*}
\frac{1}{2}n - \sqrt{\frac{5n}{2}} + o(\sqrt{n}) \leq f(n) \leq \frac{1}{2}n - \sqrt{\frac{3n}{2}} + o(\sqrt{n}).
\end{equation*}
Therefore, asymptotically, the maximum number of holes satisfies \( f(n) \approx \frac{n}{2} \). In the same paper, they have shown that \( f(n_k) = h_k \) for every \( k \geq 1 \) where
\begin{equation*}
n_k = \frac{2^{2k+1} + 3 \times 2^{k+1} + 4}{3} \quad \text{and} \quad h_k = \frac{2^{2k} - 1}{3}.
\end{equation*}
Malen and Rold\'{a}n-Roa determined the values of $f(n)$ and $g(n)$ for all positive integers $n$ in \cite{Malen}.
\par
Motivated by previous research, we study the problem of the maximal number of deep holes that an $n$-omino can enclose. By a deep hole, we assume a hole whose boundary does not have any common point with the outer boundary of the $n$-omino or with the boundary of another hole.  

\section{Holes and deep holes }

In the literature, the term \textit{hole} in a polyomino refers to a bounded, connected component of its complement in a plane. In contrast to the considerations of other mathematicians, we define\textit{deep hole} in a given polyomino as a region entirely surrounded by the cells of the polyomino. That means that the vertices of a deep hole cannot belong to  the outer boundary of the shape or the boundary of other deep holes.  

A polyomino without a hole is called \textit{free}. From a topological point of view, a free polyomino is simply connected, while a polyomino with $k$ holes has the homotopy type of a wedge of $k$ circles. Therefore, free polyominoes are suitable for tilings, and they have been studied more extensively in mathematics than those with holes. In Figure \ref{fig1}, from left to right, we show examples of two free polyominoes and a heptomino with one hole that is not deep because the vertices of a hole lie on the outer boundary of a polyomino.

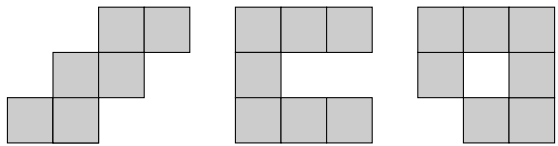
\begin{figure}[h]
\begin{center}
\begin{tikzpicture}[xscale=0.6, yscale=0.6]
    \begin{scope}
        \foreach \x/\y in {
            2/3, 3/3,
            1/2, 2/2,
            0/1, 1/1,
        } {
            \fill[gray!40!] (\x,\y) rectangle ++(1,1);
            \draw[black] (\x,\y) rectangle ++(1,1);
        }
    \end{scope}

    \begin{scope}[xshift=5cm]
        \foreach \x/\y in {
            0/3, 1/3,
            0/2, 2/3,
            0/1, 1/1, 2/1
        } {
            \fill[gray!40!] (\x,\y) rectangle ++(1,1);
            \draw[black] (\x,\y) rectangle ++(1,1);
        }
    \end{scope}

    \begin{scope}[xshift=9cm]
        \foreach \x/\y in { 1/1, 2/1, 0/2, 2/2, 0/3, 1/3, 2/3} {
            \fill[gray!40!] (\x,\y) rectangle ++(1,1);
            \draw[black] (\x,\y) rectangle ++(1,1);
        }
    \end{scope}
\end{tikzpicture}

\caption{Free hexomino and heptomino, and a heptomino with one hole}\label{fig1}
\end{center}
\end{figure}

However, the boundary of a polyomino with deep holes is a disjoint union of non-self-intersecting loops. The boundary of a free polyomino is a single loop without self-intersection, so polyominoes with deep holes have topological properties similar to those of free polyominoes. A deep hole is a hole in the sense studied by Kahle, Malen, Rold\' {a}n-Roa and others, but not every hole is a deep hole. Holes in the right polyomino in Figure \ref{fig2} are deep, whereas those in the left are not, as the vertices of a hole cannot belong to other holes.

\par 

\begin{figure}[h]
\begin{center}
\begin{tikzpicture}[xscale=0.6, yscale=0.6]
    \begin{scope}
        \foreach \x/\y in {
            0/3, 1/3, 2/3,
            0/2, 2/2,
            0/1, 1/1, 3/1,
            1/0, 2/0, 3/0
        } {
            \fill[gray!40!] (\x,\y) rectangle ++(1,1);
            \draw[black] (\x,\y) rectangle ++(1,1);
        }
    \end{scope}

    \begin{scope}[xshift=5cm]
        \foreach \x/\y in {
            0/3, 1/3, 2/3,
            0/2, 2/2, 3/2, 4/2,
            0/1, 1/1, 2/1, 4/1,
            2/0, 3/0, 4/0
        } {
            \fill[gray!40!] (\x,\y) rectangle ++(1,1);
            \draw[black] (\x,\y) rectangle ++(1,1);
        }
    \end{scope}
\end{tikzpicture}
\caption{Polyominoes with no deep holes and two deep holes}\label{fig2}
\end{center}
\end{figure}
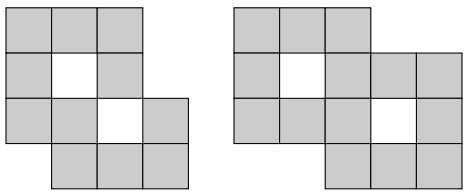
\par
Motivated by \cite{Kahle}, \cite{Roldan} and \cite{Malen}, we investigate the maximal number of deep holes, $h_n$, that can be enclosed by an $n$-omino. Polyominoes containing only deep holes form a particularly interesting subclass of polyominoes because all the boundary components are topologically circles.

It is obviously clear that we need at least eight unit squares to enclose a deep hole. Therefore, $h_n=0$ for all $1\leq n\leq 7$.

Let us make some preliminary estimations of the number of deep holes inside a polyomino. Observe that we can exclude $a b$ cells from the $(2a+1)\times (2b+1)$ polyomino. It means that for $n=3ab + 2a + 2b + 1$, we have $h_n \geq ab$, see \ref{fig3a}.

\begin{figure}[h]
\begin{center}
\begin{tikzpicture}[scale=0.4]
\foreach \row in {0,1,2} {
    \foreach \col in {0,1,2} {
        \begin{scope}[xshift=7*\col cm, yshift=-5*\row cm]
            \foreach \x/\y in {0/0,1/0,2/0,3/0,4/0, 
                               0/1,    2/1,    4/1, 
                               0/2,1/2,2/2,3/2,4/2} {
                \fill[gray!40] (\x,\y) rectangle ++(1,1);
            }
            \fill[white] (1,1) rectangle ++(1,1);
            \draw (0,0) grid (5,3);
            \draw (1,1) rectangle ++(1,1);
        \end{scope}
    }
}
\foreach \row in {0,1,2} {
    \foreach \col in {0,1} {
        \node at (7*\col + 6, -5*\row + 1.5) {\Large $\cdots$};
    }
}
\foreach \col in {0,1,2} {
    \foreach \row in {0,1} {
        \node at (7*\col + 2.5, -5*\row -1) {\Large $\vdots$};
    }
}

\node at (7*0 + 6, -5*0 -1) {\Large $\ddots$}; 
\node at (7*1 + 6, -5*0 -1) {\Large $\ddots$}; 
\node at (7*0 + 6, -5*1 -1) {\Large $\ddots$};
\node at (7*1 + 6, -5*1 -1) {\Large $\ddots$};

\end{tikzpicture}
\caption{$ab$ deep holes enclosed by a $(3ab + 2a + 2b + 1)$-omino}\label{fig3a}
\end{center}
\end{figure}
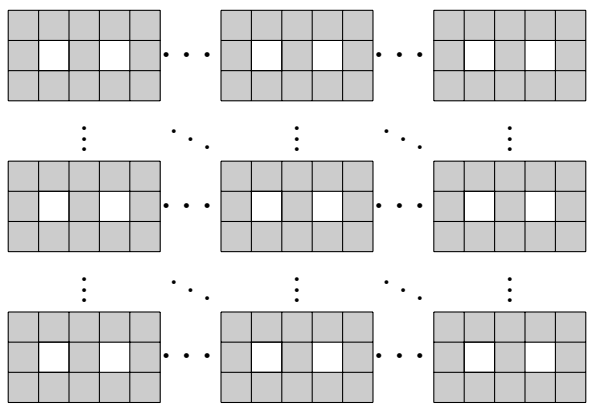

Specifially, for $n= 5k+3$ we have $h_n\geq k$. We can add up to four unit cells to a $(2k+1)\times 3$ rectangle without producing a new hole, so we obtain $h_n\geq k$ for $n=5k+4$, $5k+5$, $5k+6$, $5k+7$, see Figure \ref{fig3}. But this means that \begin{align}\label{petina}
  h_n \geq \left[\frac{n-3}{5}\right]  
\end{align}

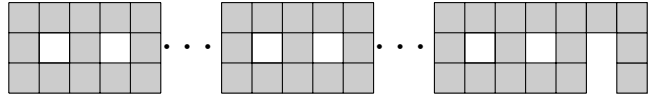
\begin{figure}[h]
\begin{center}
\begin{tikzpicture}[scale=0.4]

\foreach \row in {0} {
    \foreach \col in {0,1,2} {
        \begin{scope}[xshift=7*\col cm, yshift=-5*\row cm]
            \foreach \x/\y in {0/0,1/0,2/0,3/0,4/0, 
                               0/1,    2/1,    4/1, 
                               0/2,1/2,2/2,3/2,4/2} {
                \fill[gray!40] (\x,\y) rectangle ++(1,1);
            }
            \fill[white] (1,1) rectangle ++(1,1);
            \draw (0,0) grid (5,3);
            \draw (1,1) rectangle ++(1,1);
        \end{scope}
    }
}
\foreach \row in {0} {
    \foreach \col in {0,1} {
        \node at (7*\col + 6, -5*\row + 1.5) {\Large $\cdots$};
    }
}
\foreach \x in {19} {
    \foreach \y in {2} {
        \fill[gray!40] (\x,\y) rectangle ++(1,1);
        \draw (\x,\y) rectangle ++(1,1);
    }
}
\foreach \x in {20} {
    \foreach \y in {0,1,2} {
        \fill[gray!40] (\x,\y) rectangle ++(1,1);
        \draw(\x,\y) rectangle ++(1,1);
    }
}

\end{tikzpicture}
\caption{Polyminoes enclosing $\left[\frac{n-3}{5}\right]$ holes for $n\equiv 2 \pmod{5}$}\label{fig3}
\end{center}
\end{figure}

This bound can be improved further. But before we construct a particular example that yields a better lower bound, let us recall a classical result that we will use to obtain an upper bound for $h_n$. We will apply the celebrated Pick's Theorem, discovered by Georg Alexander Pick \cite{Pick}. It is a method used to calculate the area of a lattice polygon by counting the number of interior and boundary lattice points of the given lattice polygon in a square grid. Initially, it was used to determine the area of a simple polygon without holes and non-intersecting edges did not; however, it was later expanded to determine the area of polygons with $k$ holes, provided that their boundaries are disjoint, as seen in \cite{Funkenbusch}. 
\par
For a given simple polygon or even for a polygon with holes, whose vertices lie on lattice points with integer coordinates in a Cartesian plane, Pick's theorem can be used to calculate the area of the polygon. Pick's theorem states that if a simple polygon $P$ has vertices at lattice points and if $b$ is the number of lattice points on the boundary and $i$ is the number of lattice points in the interior, then its area is given by 

\par
$$ \mathrm{Area} (P) = i + \frac{b}{2} - 1$$
Further, if a polygon has $k$ holes with disjoint boundaries, then the area is denoted by
\par
$$ \mathrm{Area} (P) = i + \frac{b}{2} - 1 +k$$

\section{Lower Bound}

In this section, we are going to construct a sequence of $n$-ominoes  $A_n$ that contain many unit square deep holes, effectively improving the lower bound \eqref{petina}. First, we explain the algorithm for the construction of $A_n$ and find the number $f_n$ of deep holes in $A_n$.

For each positive integer $n$ there exists a unique $a$ such that
\begin{align*}
    (4a+3)^2 - (2a+1)^2 &\leq n < (4a+7)^2 - (2a+3)^2
\end{align*}
or equivalently
\begin{align}\label{ina}
    12a^2 + 20a + 8 &\leq n < 12a^2 + 44a+40.
\end{align}
The difference between the two bounds is
\begin{align*}
    12a^2 + 44a+40 - (12a^2 + 20a + 8) = 24a + 32
\end{align*}
so each positive integer $n$ can be uniquely written in the form \begin{align}
    n=(12 a^2+20a +8)+k
\end{align} where $k$ and $a$ are integers such that $0 \leq k < 24a + 32$ and $a \geq -1$.

For $n=12 a^2+20a +8$, with $a\geq 0$ we set $A_n$ to be $(4a+3)\times (4a+3)$ without $(2a+1)^2$ unit squares as in Figure \ref{s0}. Other polyominoes in the sequence will be obtained by adding $k$ unit squares to this configuration. 
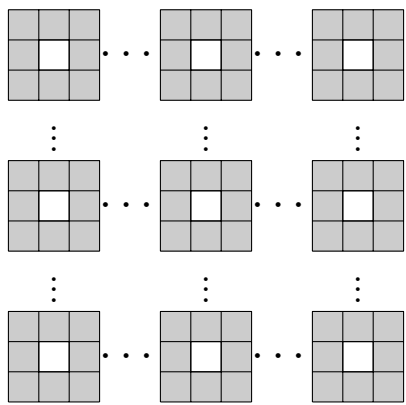
\begin{figure}[h]
\begin{center}
\begin{tikzpicture}[scale=0.4]
\foreach \row in {0,1,2} {
    \foreach \col in {0,1,2} {
        \begin{scope}[xshift=5*\col cm, yshift=-5*\row cm]
            \foreach \x/\y in {0/0,1/0,2/0,0/1,2/1,0/2,1/2,2/2} {
                \fill[gray!40] (\x,\y) rectangle ++(1,1);
            }
            \fill[white] (1,1) rectangle ++(1,1);
            \draw (0,0) grid (3,3);
            \draw (1,1) rectangle ++(1,1);
        \end{scope}
    }
}
\foreach \row in {0,1,2} {
    \foreach \col in {0,1} {
        \node at (5*\col + 4, -5*\row + 1.5) {\Large $\cdots$};
    }
}
\foreach \col in {0,1,2} {
    \foreach \row in {0,1} {
        \node at (5*\col + 1.5, -5*\row -1) {\Large $\vdots$};
    }
}
\end{tikzpicture}
\caption{Polyomino $A_n$ for $n=12a^2 + 20a + 8$}\label{s0}
\end{center}
\end{figure}

\begin{construction}\label{qn}
Depending on $0\leq k<24a+32$, the polyomino $A_n$ is constructed as follows
\begin{itemize} 
    \item If $0\leq k \leq 4$, $A_n$ is $(4a+3)\times (4a+3)$ without $(2a+1)^2$  unit squares with $k$ additional unit squares on the right side as indicated in Figure \ref{s1} for $k=4$. For these polyominoes $f_n=(2a+1)^2$
\begin{figure}[h!]
\begin{center}
\begin{tikzpicture}[scale=0.4]
\foreach \row in {0,1,2} {
    \foreach \col in {0,1,2} {
        \begin{scope}[xshift=5*\col cm, yshift=-5*\row cm]
            \foreach \x/\y in {0/0,1/0,2/0,0/1,2/1,0/2,1/2,2/2} {
                \fill[gray!40] (\x,\y) rectangle ++(1,1);
            }
            \fill[white] (1,1) rectangle ++(1,1);
            \draw (0,0) grid (3,3);
            \draw (1,1) rectangle ++(1,1);
        \end{scope}
    }
}
\foreach \row in {0,1,2} {
    \foreach \col in {0,1} {
        \node at (5*\col + 4, -5*\row + 1.5) {\Large $\cdots$};
    }
}
\foreach \col in {0,1,2} {
    \foreach \row in {0,1} {
        \node at (5*\col + 1.5, -5*\row -1) {\Large $\vdots$};
    }
}
\node at (5*0 + 4, -5*0 -1) {\Large $\ddots$}; 
\node at (5*1 + 4, -5*0 -1) {\Large $\ddots$}; 
\node at (5*0 + 4, -5*1 -1) {\Large $\ddots$};
\node at (5*1 + 4, -5*1 -1) {\Large $\ddots$};
\foreach \x in {13} {
    \foreach \y in {0} {
        \fill[gray!20] (\x,\y) rectangle ++(1,1);
        \draw[dotted] (\x,\y) rectangle ++(1,1);
    }
}
\foreach \x in {14} {
    \foreach \y in {0, 1,2} {
        \fill[gray!20] (\x,\y) rectangle ++(1,1);
        \draw[dotted] (\x,\y) rectangle ++(1,1);
    }
}

\end{tikzpicture}
\caption{Polyomino $A_n$ for $n=12a^2+20a+12$}\label{s1}
\end{center}
\end{figure}
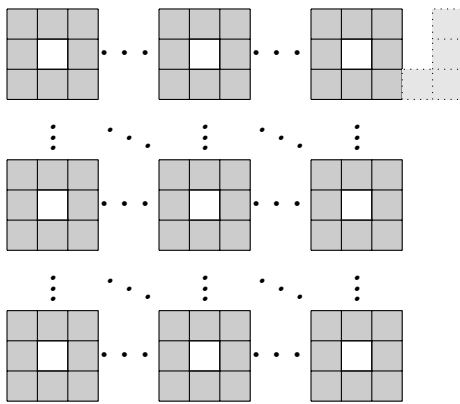

\item If $5\leq k \leq 7$  then $A_n$ is $(4a+3)\times (4a+3)$ without $(2a+1)^2$ unit squares  with a `bridge' of five unit squares on the right side with remaining $k-5$ cells trying to enclose one additional deep hole, see Figure \ref{s2} for $k=5$. For these $A_n$, $f_n=(2a+1)^2+1$.    
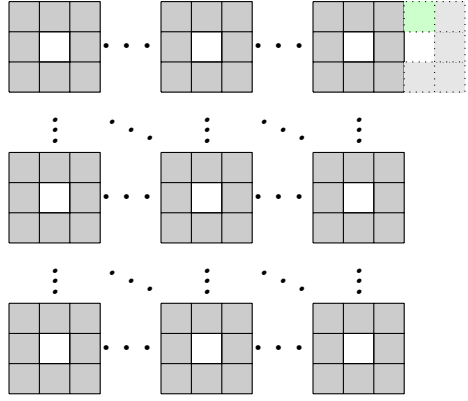
\begin{figure}[h]
\begin{center}
\begin{tikzpicture}[scale=0.4]
\foreach \row in {0,1,2} {
    \foreach \col in {0,1,2} {
        \begin{scope}[xshift=5*\col cm, yshift=-5*\row cm]
            \foreach \x/\y in {0/0,1/0,2/0,0/1,2/1,0/2,1/2,2/2} {
                \fill[gray!40] (\x,\y) rectangle ++(1,1);
            }
            \fill[white] (1,1) rectangle ++(1,1);
            \draw (0,0) grid (3,3);
            \draw (1,1) rectangle ++(1,1);
        \end{scope}
    }
}
\foreach \row in {0,1,2} {
    \foreach \col in {0,1} {
        \node at (5*\col + 4, -5*\row + 1.5) {\Large $\cdots$};
    }
}
\foreach \col in {0,1,2} {
    \foreach \row in {0,1} {
        \node at (5*\col + 1.5, -5*\row -1) {\Large $\vdots$};
    }
}
\node at (5*0 + 4, -5*0 -1) {\Large $\ddots$}; 
\node at (5*1 + 4, -5*0 -1) {\Large $\ddots$}; 
\node at (5*0 + 4, -5*1 -1) {\Large $\ddots$};
\node at (5*1 + 4, -5*1 -1) {\Large $\ddots$};
\foreach \x in {13} {
    \foreach \y in {0,2} {
        \fill[gray!20] (\x,\y) rectangle ++(1,1);
        \draw[dotted] (\x,\y) rectangle ++(1,1);
    }
}
\foreach \x in {14} {
    \foreach \y in {0,1,2} {
        \fill[gray!20] (\x,\y) rectangle ++(1,1);
        \draw[dotted] (\x,\y) rectangle ++(1,1);
    }
}
\foreach \x in {13} {
    \foreach \y in {2} {
        \fill[green!20] (\x,\y) rectangle ++(1,1);
        \draw[dotted] (\x,\y) rectangle ++(1,1);
    }
}
\end{tikzpicture}
\caption{Polyomino $A_n$ for $n=12a^2+20a+13$}\label{s2}
\end{center}
\end{figure}
\item If $8\leq k \leq 6a+5$ then $A_n$ is $(4a+3)\times (4a+3)$ without $(2a+1)^2$  unit squares with $k$ additional unit squares on its  right side arranged to enclose as many unit squares on the right side as it is indicated in Figure \ref{s3} for $k=6a+5$. Then the corresponding $f_n=(2a+1)^2 + \left\lfloor \frac{(k-2)}{3} \right\rfloor$
    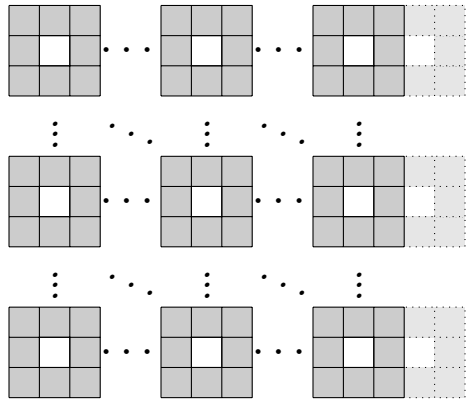
\begin{figure}[h]
\begin{center}
\begin{tikzpicture}[scale=0.4]
\foreach \row in {0,1,2} {
    \foreach \col in {0} {
        \begin{scope}[xshift=5*\col cm, yshift=-5*\row cm]
            \foreach \x/\y in {0/0,1/0,2/0,0/1,2/1,0/2,1/2,2/2}{
                \fill[gray!40] (\x,\y) rectangle ++(1,1);
            }
            \fill[white] (1,1) rectangle ++(1,1);
            \draw (0,0) grid (3,3);
            \draw (1,1) rectangle ++(1,1);
            \foreach \x in {13} {
                \foreach \y in {0,2} {
                \fill[gray!20] (\x,\y) rectangle ++(1,1);
                \draw[dotted] (\x,\y) rectangle ++(1,1);
                }
            }
            \foreach \x in {14} {
                \foreach \y in {0,1,2} {
                \fill[gray!20] (\x,\y) rectangle ++(1,1);
                \draw[dotted] (\x,\y) rectangle ++(1,1);
                }
            }
        \end{scope}
    }
}
\foreach \row in {0,1,2} {
        \foreach \col in {1,2} {
        \begin{scope}[xshift=5*\col cm, yshift=-5*\row cm]
            \foreach \x/\y in {0/0,1/0,2/0,0/1,2/1,0/2,1/2,2/2}{
                \fill[gray!40] (\x,\y) rectangle ++(1,1);
            }
            \fill[white] (1,1) rectangle ++(1,1);
            \draw (0,0) grid (3,3);
            \draw (1,1) rectangle ++(1,1);
        \end{scope}
    }
}
\foreach \row in {0,1,2} {
    \foreach \col in {0,1} {
        \node at (5*\col + 4, -5*\row + 1.5) {\Large $\cdots$};
    }
}
\foreach \col in {0,1,2} {
    \foreach \row in {0,1} {
        \node at (5*\col + 1.5, -5*\row -1) {\Large $\vdots$};
    }
}
\node at (5*0 + 4, -5*0 -1) {\Large $\ddots$}; 
\node at (5*1 + 4, -5*0 -1) {\Large $\ddots$}; 
\node at (5*0 + 4, -5*1 -1) {\Large $\ddots$};
\node at (5*1 + 4, -5*1 -1) {\Large $\ddots$};
\end{tikzpicture}
\caption{Polyomino $A_n$ for $n=12a^2+26a+13$}\label{s3}  
\end{center}
\end{figure}
    \item If $6a+6\leq k \leq 6a+9$ then $A_n$ is $(4a+5)\times (4a+3)$ with $(2a+1)\times (2a+2)$  unit squares removed and  with $k-6a-5$ unit squares attached above its  upper side, see Figure \ref{s4}. In this case, $f_n=(2a+1)^2+2a+1$.

\begin{figure}[h]
\begin{center}
\begin{tikzpicture}[scale=0.4]
\foreach \row in {0,1,2} {
    \foreach \col in {0} {
        \begin{scope}[xshift=5*\col cm, yshift=-5*\row cm]
            \foreach \x/\y in {0/0,1/0,2/0,0/1,2/1,0/2,1/2,2/2}{
                \fill[gray!40] (\x,\y) rectangle ++(1,1);
            }
            \fill[white] (1,1) rectangle ++(1,1);
            \draw (0,0) grid (3,3);
            \draw (1,1) rectangle ++(1,1);
            \foreach \x in {13} {
                \foreach \y in {0,2} {
                \fill[gray!20] (\x,\y) rectangle ++(1,1);
                \draw[dotted] (\x,\y) rectangle ++(1,1);
                }
            }
            \foreach \x in {14} {
                \foreach \y in {0,1,2} {
                \fill[gray!20] (\x,\y) rectangle ++(1,1);
                \draw[dotted] (\x,\y) rectangle ++(1,1);
                }
            }
        \end{scope}
    }
}
\foreach \row in {0,1,2} {
        \foreach \col in {1,2} {
        \begin{scope}[xshift=5*\col cm, yshift=-5*\row cm]
            \foreach \x/\y in {0/0,1/0,2/0,0/1,2/1,0/2,1/2,2/2}{
                \fill[gray!40] (\x,\y) rectangle ++(1,1);
            }
            \fill[white] (1,1) rectangle ++(1,1);
            \draw (0,0) grid (3,3);
            \draw (1,1) rectangle ++(1,1);
        \end{scope}
    }
}
\foreach \row in {0,1,2} {
    \foreach \col in {0,1} {
        \node at (5*\col + 4, -5*\row + 1.5) {\Large $\cdots$};
    }
}
\foreach \col in {0,1,2} {
    \foreach \row in {0,1} {
        \node at (5*\col + 1.5, -5*\row -1) {\Large $\vdots$};
    }
}
\foreach \y in {3} {
    \foreach \x in {10,12} {
    \fill[gray!20] (\x,\y) rectangle ++(1,1);
    \draw[dotted] (\x,\y) rectangle ++(1,1);
    }
}
\foreach \y in {4} {
    \foreach \x in {10,12} {
    \fill[gray!20] (\x,\y) rectangle ++(1,1);
    \draw[dotted] (\x,\y) rectangle ++(1,1);
    }
}
\node at (5*0 + 4, -5*0 -1) {\Large $\ddots$}; 
\node at (5*1 + 4, -5*0 -1) {\Large $\ddots$}; 
\node at (5*0 + 4, -5*1 -1) {\Large $\ddots$};
\node at (5*1 + 4, -5*1 -1) {\Large $\ddots$};
\end{tikzpicture}
\caption{Polyomino $A_n$ for $n=12a^2+26a+17$}\label{s4}
\end{center}
\end{figure}
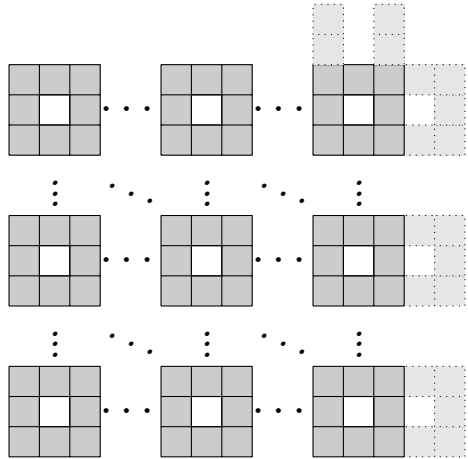

 \item If $6a+10\leq k\leq 6a+12$ then $A_n$ is $(4a+5)\times (4a+3)$ without $(2a+1)\times (2a+2)$  unit squares with 5 unit squares forming `bridge' over the upper side and the remaining cells arranged to enclose an additional unit square deep hole, see Figure \ref{s5}. It enclose $f_n=(2a+1)^2 +2a+2$ deep holes.
 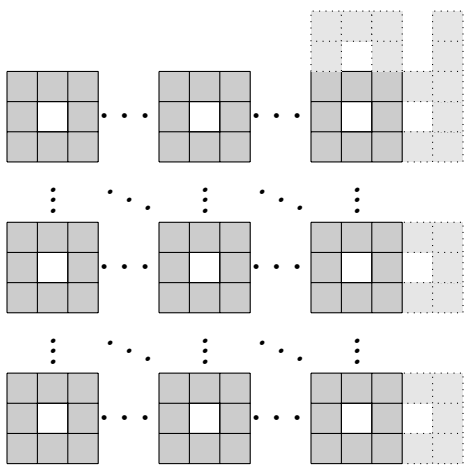
\begin{figure}[h]
\centering
\begin{tikzpicture}[scale=0.4]
\foreach \row in {0,1,2} {
    \foreach \col in {0} {
        \begin{scope}[xshift=5*\col cm, yshift=-5*\row cm]
            \foreach \x/\y in {0/0,1/0,2/0,0/1,2/1,0/2,1/2,2/2}{
                \fill[gray!40] (\x,\y) rectangle ++(1,1);
            }
            \fill[white] (1,1) rectangle ++(1,1);
            \draw (0,0) grid (3,3);
            \draw (1,1) rectangle ++(1,1);
            \foreach \x in {13} {
                \foreach \y in {0,2} {
                \fill[gray!20] (\x,\y) rectangle ++(1,1);
                \draw[dotted] (\x,\y) rectangle ++(1,1);
                }
            }
            \foreach \x in {14} {
                \foreach \y in {0,1,2} {
                \fill[gray!20] (\x,\y) rectangle ++(1,1);
                \draw[dotted] (\x,\y) rectangle ++(1,1);
                }
            }
        \end{scope}
    }
}
\foreach \row in {0,1,2} {
        \foreach \col in {1,2} {
        \begin{scope}[xshift=5*\col cm, yshift=-5*\row cm]
            \foreach \x/\y in {0/0,1/0,2/0,0/1,2/1,0/2,1/2,2/2}{
                \fill[gray!40] (\x,\y) rectangle ++(1,1);
            }
            \fill[white] (1,1) rectangle ++(1,1);
            \draw (0,0) grid (3,3);
            \draw (1,1) rectangle ++(1,1);
        \end{scope}
    }
}
\foreach \row in {0,1,2} {
    \foreach \col in {0,1} {
        \node at (5*\col + 4, -5*\row + 1.5) {\Large $\cdots$};
    }
}
\foreach \col in {0,1,2} {
    \foreach \row in {0,1} {
        \node at (5*\col + 1.5, -5*\row -1) {\Large $\vdots$};
    }
}
\foreach \y in {3} {
    \foreach \x in {10,12,14} {
    \fill[gray!20] (\x,\y) rectangle ++(1,1);
    \draw[dotted] (\x,\y) rectangle ++(1,1);
    }
}
\foreach \y in {4} {
    \foreach \x in {10,11,12,14} {
    \fill[gray!20] (\x,\y) rectangle ++(1,1);
    \draw[dotted] (\x,\y) rectangle ++(1,1);
    }
}
\node at (5*0 + 4, -5*0 -1) {\Large $\ddots$}; 
\node at (5*1 + 4, -5*0 -1) {\Large $\ddots$}; 
\node at (5*0 + 4, -5*1 -1) {\Large $\ddots$};
\node at (5*1 + 4, -5*1 -1) {\Large $\ddots$};
\end{tikzpicture}
\caption{Polyomino $A_n$ for $n=12a^2+26a+20$}\label{s5}
\end{figure}

    \item If $6a+13\leq k \leq 12 a+13$ then $A_n$ is $(4a+5)\times (4a+3)$ without $(2a+1)\times (2a+2)$  unit squares  with $k-6a-5$ squares arranged along the upper side to enclose as many unit squares deep holes, as in Figure \ref{s6}. Then $f_n=(2a+1)^2 +2 +\left\lfloor \frac{k-10}{3} \right\rfloor$ 
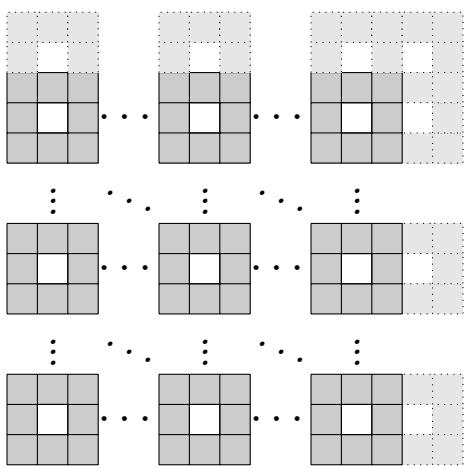
\begin{figure}[h]
\centering
\begin{tikzpicture}[scale=0.4]
\foreach \row in {0,1,2} {
    \foreach \col in {0} {
        \begin{scope}[xshift=5*\col cm, yshift=-5*\row cm]
            \foreach \x/\y in {0/0,1/0,2/0,0/1,2/1,0/2,1/2,2/2}{
                \fill[gray!40] (\x,\y) rectangle ++(1,1);
            }
            \fill[white] (1,1) rectangle ++(1,1);
            \draw (0,0) grid (3,3);
            \draw (1,1) rectangle ++(1,1);
            \foreach \x in {13} {
                \foreach \y in {0,2} {
                \fill[gray!20] (\x,\y) rectangle ++(1,1);
                \draw[dotted] (\x,\y) rectangle ++(1,1);
                }
            }
            \foreach \x in {14} {
                \foreach \y in {0,1,2} {
                \fill[gray!20] (\x,\y) rectangle ++(1,1);
                \draw[dotted] (\x,\y) rectangle ++(1,1);
                }
            }
        \end{scope}
    }
}
\foreach \row in {0,1,2} {
        \foreach \col in {1,2} {
        \begin{scope}[xshift=5*\col cm, yshift=-5*\row cm]
            \foreach \x/\y in {0/0,1/0,2/0,0/1,2/1,0/2,1/2,2/2}{
                \fill[gray!40] (\x,\y) rectangle ++(1,1);
            }
            \fill[white] (1,1) rectangle ++(1,1);
            \draw (0,0) grid (3,3);
            \draw (1,1) rectangle ++(1,1);
        \end{scope}
    }
}
\foreach \row in {0,1,2} {
    \foreach \col in {0,1} {
        \node at (5*\col + 4, -5*\row + 1.5) {\Large $\cdots$};
    }
}
\foreach \col in {0,1,2} {
    \foreach \row in {0,1} {
        \node at (5*\col + 1.5, -5*\row -1) {\Large $\vdots$};
    }
}
\foreach \y in {3} {
    \foreach \x in {0,2,5,7,10,12,14} {
    \fill[gray!20] (\x,\y) rectangle ++(1,1);
    \draw[dotted] (\x,\y) rectangle ++(1,1);
    }
}
\foreach \y in {4} {
    \foreach \x in {0,1,2,5,6,7,10,11,12,13,14} {
    \fill[gray!20] (\x,\y) rectangle ++(1,1);
    \draw[dotted] (\x,\y) rectangle ++(1,1);
    }
}
\node at (5*0 + 4, -5*0 -1) {\Large $\ddots$}; 
\node at (5*1 + 4, -5*0 -1) {\Large $\ddots$}; 
\node at (5*0 + 4, -5*1 -1) {\Large $\ddots$};
\node at (5*1 + 4, -5*1 -1) {\Large $\ddots$};
\end{tikzpicture}
\caption{Polyomino $A_n$ for $n=12a^2+32a+31$}\label{s6}
\end{figure}
    
    \item If $12a+14\leq k\leq 12 a +17$ then $A_n$ is $(4a+5)\times (4a+5)$ without $(2a+2)^2$  unit squares with $k-12 a-13$ unit squares on the left side. Here $f_n=(2a+1)^2+4a+3$.
    \item If $12a+18\leq k \leq12a+ 20$ then $A_n$ is  $(4a+5)\times (4a+5)$ without $(2a+2)^2$  unit squares, with five unit squares forming `bridge' on the left side and remaining unit squares arranged to enclose one additional unit square. In this case $f_n=(2a+1)^2+4a+4$.
    \item If $12a+21 \leq k \leq 18a +21$ then $A_n$ is  $(4a+5)\times (4a+5)$ without $(2a+2)^2$  unit squares with $k-12a-13$ unit squares arranged to enclose as many unit square deep holes as they can on its left side as in Figure \ref{s7}. The corresponding value $f_n=(2a+1)^2 + 4 +\left\lfloor \frac{k-18}{3} \right\rfloor$.
\begin{figure}[h]
\centering
\begin{tikzpicture}[xscale=0.4, yscale=0.4]
    \draw[step=1cm,white] (0,0) grid (16,15);
    \node at (6, 4) {\Large $\ddots$};
    \node at (11, 4) {\Large $\ddots$};
    \node at (6, 9) {\Large $\ddots$};
    \node at (11, 9) {\Large $\ddots$};
    \node at (6,1.5) {\Large $\cdots$};
    \node at (11,1.5) {\Large $\cdots$};
    \node at (6,6.5) {\Large $\cdots$};
    \node at (11,6.5) {\Large $\cdots$};
    \node at (6,11.5) {\Large $\cdots$};
    \node at (11,11.5) {\Large $\cdots$};
    \node at (3.5,4) {\Large $\vdots$};
    \node at (8.5,4) {\Large $\vdots$};
    \node at (13.5,4) {\Large $\vdots$};
    \node at (3.5,9) {\Large $\vdots$};
    \node at (8.5,9) {\Large $\vdots$};
    \node at (13.5,9) {\Large $\vdots$};
    \begin{scope}
    \foreach \x/\y in {
        2/0, 3/0, 4/0, 7/0, 8/0, 9/0, 12/0, 13/0, 14/0,
        2/1, 4/1, 7/1, 9/1, 12/1, 14/1,
        2/2, 3/2, 4/2, 7/2, 8/2, 9/2, 12/2, 13/2, 14/2,
        2/5, 3/5, 4/5, 7/5, 8/5, 9/5, 12/5,13/5,14/5,
        2/6, 4/6, 7/6, 9/6, 12/6, 14/6, 
        2/7, 3/7, 4/7, 7/7, 8/7, 9/7, 12/7, 13/7, 14/7,
        2/10, 3/10, 4/10, 7/10, 8/10, 9/10, 12/10, 13/10, 14/10,
        2/11, 4/11, 7/11, 9/11, 12/11, 14/11, 
        2/12, 3/12, 4/12, 7/12, 8/12, 9/12, 12/12, 13/12, 14/12
    } {
        \fill[gray!40] (\x,\y) rectangle ++(1,1);
        \draw (\x,\y) rectangle ++(1,1);
    }
    
    \foreach \x/\y in {3/1, 8/1, 13/1, 3/6, 8/6, 13/6, 3/11, 8/11, 13/11,} {
        \fill[white] (\x,\y) rectangle ++(1,1);
        \draw (\x,\y) rectangle ++(1,1);
    }
    \foreach \x/\y in {
        0/0, 1/0, 15/0, 16/0, 
        0/1, 16/1, 
        0/2, 1/2, 15/2, 16/2,
        0/5, 1/5, 15/5, 16/5,
        0/6, 16/6,
        0/7, 1/7, 15/7, 16/7,
        0/10, 1/10, 15/10, 16/10, 
        0/11, 16/11, 
        0/12, 1/12, 15/12, 16/12, 
        0/14, 2/13, 4/13, 7/13, 9/13, 12/13, 14/13, 16/13,
              1/14, 2/14, 3/14, 4/14, 7/14, 8/14, 9/14, 12/14, 13/14, 14/14, 15/14, 16/14
        } {
        \fill[gray!20] (\x,\y) rectangle ++(1,1);
        \draw[dotted] (\x,\y) rectangle ++(1,1);
    }
    \foreach \x/\y in {
        1/1
        } {
        \fill[white] (\x,\y) rectangle ++(1,1);
        \draw[dotted] (\x,\y) rectangle ++(1,1);
    }
    \end{scope}
\end{tikzpicture}
\caption{Polyomino $A_n$ for $n=12a^2+38a+29$}\label{s7}
\end{figure}
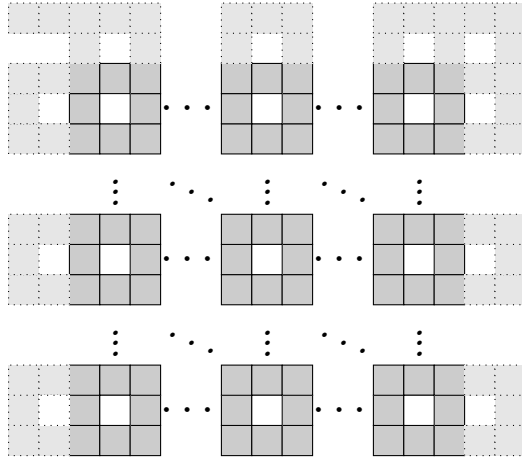

    \item If $18a+22\leq k \leq 18a+25$ then $A_n$ is $(4a+7)\times (4a+5)$ without $(2a+3)\times (2a+2)$  unit squares with $k-18a-21$ unit squares attached on the bottom trying to enclose one  additional unit square. Obviously, $f_n=(2a+1)^2+6a+5$
    \item If $18a+26\leq k \leq 18a+28$ then $A_n$ is $(4a+7)\times (4a+5)$ without $(2a+3)\times (2a+2)$  unit squares with five unit squares forming `bridge' along the bottom side and the remaining cells trying to enclose an additional unit square. Then, $f_n=(2a+1)^2+6a+6$
    \item If $18a+29 \leq k \leq 24a +31$ then $A_n$ is  $(4a+7)\times (4a+5)$ without $(2a+2)^2$  unit squares with $k-18a-21$ unit squares arranged to enclose as much unit square deep holes on the bottom side as it is depicted in Figure \ref{s8}. In this case, $A_n$ encloses $f_n=(2a+1)^2 + 6 +\left\lfloor \frac{k-26}{3} \right\rfloor$
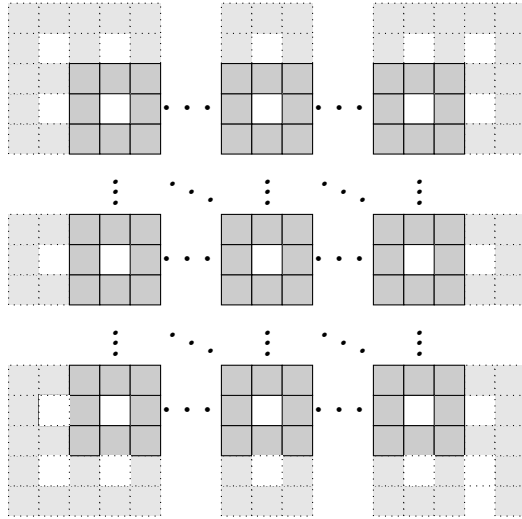
\begin{figure}[ht!]
\centering
\begin{tikzpicture}[xscale=0.4, yscale=0.4]
    \draw[step=1cm, white] (0,0) grid (16,15);
    \node at (6, 4) {\Large $\ddots$};
    \node at (11, 4) {\Large $\ddots$};
    \node at (6, 9) {\Large $\ddots$};
    \node at (11, 9) {\Large $\ddots$};
    \node at (6,1.5) {\Large $\cdots$};
    \node at (11,1.5) {\Large $\cdots$};
    \node at (6,6.5) {\Large $\cdots$};
    \node at (11,6.5) {\Large $\cdots$};
    \node at (6,11.5) {\Large $\cdots$};
    \node at (11,11.5) {\Large $\cdots$};
    \node at (3.5,4) {\Large $\vdots$};
    \node at (8.5,4) {\Large $\vdots$};
    \node at (13.5,4) {\Large $\vdots$};
    \node at (3.5,9) {\Large $\vdots$};
    \node at (8.5,9) {\Large $\vdots$};
    \node at (13.5,9) {\Large $\vdots$};
    \begin{scope}
    \foreach \x/\y in {
        0/-2, 1/-2, 2/-2, 3/-2, 4/-2, 7/-2, 8/-2, 9/-2, 12/-2, 13/-2, 14/-2, 16/-2, 0/-1, 2/-1, 4/-1, 7/-1, 9/-1,12/-1, 14/-1, 16/-1,
        0/0, 1/0, 15/0, 16/0, 
        0/1, 16/1, 
        0/2, 1/2, 15/2, 16/2,
        0/5, 1/5, 15/5, 16/5,
        0/6, 16/6,
        0/7, 1/7, 15/7, 16/7,
        0/10, 1/10, 15/10, 16/10, 
        0/11, 16/11, 
        0/12, 1/12, 15/12, 16/12, 
        0/13, 2/13, 4/13, 7/13, 9/13, 12/13, 14/13, 16/13,
        0/14, 1/14, 2/14, 3/14, 4/14, 7/14, 8/14, 9/14, 12/14, 13/14, 14/14, 15/14, 16/14
        } {
        \fill[gray!20] (\x,\y) rectangle ++(1,1);
        \draw[dotted] (\x,\y) rectangle ++(1,1);
        }
    \foreach \x/\y in {
        2/0, 3/0, 4/0, 7/0, 8/0, 9/0, 12/0, 13/0, 14/0,
        2/1, 4/1, 7/1, 9/1, 12/1, 14/1,
        2/2, 3/2, 4/2, 7/2, 8/2, 9/2, 12/2, 13/2, 14/2,
        2/5, 3/5, 4/5, 7/5, 8/5, 9/5, 12/5,13/5,14/5,
        2/6, 4/6, 7/6, 9/6, 12/6, 14/6, 
        2/7, 3/7, 4/7, 7/7, 8/7, 9/7, 12/7, 13/7, 14/7,
        2/10, 3/10, 4/10, 7/10, 8/10, 9/10, 12/10, 13/10, 14/10,
        2/11, 4/11, 7/11, 9/11, 12/11, 14/11, 
        2/12, 3/12, 4/12, 7/12, 8/12, 9/12, 12/12, 13/12, 14/12, 
    } {
        \fill[gray!40] (\x,\y) rectangle ++(1,1);
        \draw (\x,\y) rectangle ++(1,1);
    }
    \foreach \x/\y in {
    3/1, 8/1, 13/1, 3/6, 8/6, 13/6, 3/11, 8/11, 13/11} {
        \fill[white] (\x,\y) rectangle ++(1,1);
        \draw (\x,\y) rectangle ++(1,1);
    }
    \foreach \x/\y in {
    1/1, 1/-1, 3/-1, 8/-1, 13/-1,15/-1 }{
        \fill[white] (\x,\y) rectangle ++(1,1);
        \draw[dotted] (\x,\y) rectangle ++(1,1);
        }
    \end{scope}
    
\end{tikzpicture}
\caption{Polyomino $A_n$ for $n=12a^2+44a+39$}\label{s8}

\end{figure}
\end{itemize}
\end{construction}

Obviously, for all $n$
\begin{equation}\label{lw}
    h_n\geq f_n
\end{equation}

Therefore we can use the explicit form of the sequence $f_n$ to deduce the following result

\begin{proposition}\label{prop1} For $ n=(12 a^2+20a +8)+k$,
 where $k$ is an integer such that $0 \leq k < 24a + 32$, it holds \[f_n\geq (2a+1)^2 \] In particular, if $k\geq 9$ then \[f_n \geq (2a+1)^2+\frac{k}{3}-3\]
    
\end{proposition}

The inequality \eqref{ina} in $a$ has the unique integer solution  
\begin{align}
    a & = \left[\frac{\sqrt{3n+1}-5}{6}\right]
\end{align}

Combining the inequality \eqref{lw} and Proposition \ref{prop1} with the classical inequality $[x]>x-1$, we establish the following lower bound 
\begin{align}\label{lb}
    h_n & > \frac{n}{3}-\frac{16}{9} \sqrt{3n+1}+\frac{65}{9}
\end{align}

Construction \ref{qn}, as we will see later, provides the exact values of $h_n$ in many cases. Therefore, the above lower bound explains the asymptotic behaviour of this sequence.

\section{Upper Bound}

A polyomino $P$ can be considered as a finite subset of the square grid. It is a special kind of polygon whose vertices lie at the lattice points and the edges along the grid lines. Pick's theorem, a celebrated result from elementary geometric combinatorics holds for such polygons, including those that are not simple connected. Recall that by generalized Pick's formula \cite{Funkenbusch} for a polyomino containing only deep holes, it holds 
\begin{align} \label{pick} n=i+\frac{b}{2}-1+h_n
\end{align}
where $i$ denotes the number of the lattice points in the interior and $b$ on the boundary of the polyomino $P$. 

The relation \eqref{pick} will be used to obtain an upper bound  on the number of deep holes $h_n$. We start with the following observation.

\begin{proposition}\label{p1}
It holds that 
\begin{align}
b \geq 4 h_n + \tilde{b} 
\end{align}
where $\tilde{b}$ is the number of the lattice points on the outer boundary of the polyomino $P$.
\end{proposition}

\begin{proof}
 The total number of lattice points on the boundary is the sum of the number of the lattice points on the boundaries of the deep holes and the number of the lattice points on the outer boundary of $P$. Since each deep hole has at least four lattice points on its boundary, the inequality follows.  
\end{proof}

\begin{proposition}\label{p2}
  It holds that 
\begin{align}\label{nej}
n+1 \geq 3 h_n + \frac{\tilde{b}}{2}
\end{align}  
\end{proposition}

\begin{proof}
 The minimum possible number of interior lattice points of $P$ is zero.  Using equation \eqref{pick} and Proposition \ref{p1}, yields the stated inequality.
\end{proof}

Let $\bar{P}$ be the union of the $n$-omino $P$ and its deep holes. Then $\bar{P}$ is a simple connected polyomino with area at least $n+h_n$. The perimeter of $\bar{P}$ is equal to $\tilde{b}$. Now we can apply the classical isoperimetric  inequality, first proved rigorously by Schwartz in \cite{Schwartz}. It implies that

\begin{align}\label{iso1}
    {\tilde{b}}^2\geq 4\pi \, \mathrm{Area} (\bar{P}) \geq 4\pi (n+h_n)
\end{align}

Since $\tilde{b}$ and $n+h_n$ are positive integers, we actually have the strict inequality.

\begin{align}\label{iso2}
    \tilde{b}>2\sqrt{(n+h_n)\pi}
\end{align}

Combining inequalities \eqref{nej} and \eqref{iso2} we obtain the following result

\begin{proposition} It holds that
    \begin{align}\label{trecina}
        n+1>3h_n+\sqrt{(n+h_n)\pi}
    \end{align}
\end{proposition}

From \eqref{trecina} we immediately obtain $n\geq 3 h_n$, but we are going to improve this inequality.

\begin{theorem} It holds that 
\begin{align}\label{upper}
h_n< \frac{6n+6+\pi-\sqrt{48n \pi+12 \pi+\pi^2}}{18}   
\end{align}
\end{theorem}

\begin{proof} We solve the inequality for $h_n$, that is,   \[n+1-3h_n >\sqrt{(n+h_n)\pi}\] Having in mind that $n\geq 3 h_n$, we square both sides. After squaring and rearranging,  we obtain a quadratic inequality  \[9h_n^2-(6n+6+\pi)h_n+(n+1)^2-n\pi>0\] Taking the above-mentioned constraints in account and solving the inequality we yield the desired bound. 
\end{proof}

Therefore, we obtain the following upper bound 
\begin{align}\label{ub}
    h_n < \frac{n}{3}-\frac{1}{18}\sqrt{48n \pi+\pi^2+12 \pi}+\frac{6+\pi}{18}
\end{align}

\section{Asymptotics of $h_n$}

In the previous sections, we established the inequalities \eqref{lb} and \eqref{ub}. Together, they give the following bounds for $h_n$.
freerer
\begin{equation}\label{lub}
      \frac{n}{3}-\frac{16}{9} \sqrt{3n+1}+o(\sqrt{n})<  h_n < \frac{n}{3}-\frac{1}{18}\sqrt{48n \pi+\pi^2+12 \pi}+o(\sqrt{n})
\end{equation}

Dividing by $n$, we obtain 

\begin{equation}\label{lub1}
      \frac{1}{3}-\frac{16}{9} \sqrt{\frac{3}{n}+\frac{1}{n^2}}+o\left(\frac{1}{\sqrt{n}}\right)<  \frac{h_n}{n} < \frac{1}{3}-\frac{1}{18}\sqrt{\frac{48\pi}{n}+\frac{\pi^2+12 \pi}{n^2}}+o\left(\frac{1}{\sqrt{n}}\right)
\end{equation}

From here, as $n\to +\infty$ by Two Policemen Lemma, we obtain the following result

\begin{theorem} For sufficiently large $n$, the maximal number of deep holes in an $n$-omino satisfies $$h_n\approx\frac{n}{3}.$$
    
\end{theorem}

The above result is analogous to Rold\'{a}n-Roa's result \cite{Roldan} for holes, that is maximal number of holes grows asymptotically $\frac{n}{2}$.

\section{Exact Value}

In this section, we refine the upper bound established in Section 4. This refinement enables us to determine the exact value of $h_n$ in some instances.

Consider an $n$-omino with deep holes. Among its edges, we distinguish those belonging to the outer perimeter, to the inner perimeter, and to the inner edges. Let the number of each of these edges be $s_o$, $s_i$, and $c$, respectively. Then it holds

\begin{align}\label{njed}
    4n & =s_o+s_i+2 c
\end{align}

If an $n$-omino encloses $h_n$ deep holes, then by the result of Harary and Harborth in \cite{harary}, it follows that $s_o \geq 2\left \lceil 2\sqrt{n+h_n}\right \rceil$. Furthermore, since the perimeter of a deep hole is at least four, $s_i \geq 4 h_n$.

Next, we estimate the number of internal edges. A polygon whose angles are only $90^\circ$ and $270^\circ$ has at least four right angles.  Each of the right angles on the boundary of a deep hole is the endpoint of two internal edges. The boundary lattice point on the outer perimeter that is not a vertex is the endpoint of exactly one internal edge. A vertex of a right angle on the outer perimeter is not the endpoint of an internal edge, whereas a vertex of $270^\circ$  is the endpoint of two internal edges. Therefore, 
\begin{align}\label{nit}
    2c & \geq s_o-4+8h_n
\end{align}

Combining \eqref{njed} and \eqref{nit} we establish the following sequence of inequalities

\begin{equation}
    4n\geq 2 s_o+s_i+8 h_n-4\geq 4 \left \lceil 2\sqrt{n+h_n}\right \rceil+12 h_n-4
\end{equation}

From the above inequality, it follows that 
\begin{equation}\label{imp}
    n\geq 3 h_n+ 2\sqrt{n+h_n}-1
\end{equation}

\begin{proposition} \begin{align}\label{bup}
    h_n &\leq \frac{n}{3}-\frac{4}{9} \sqrt{3n+1}+\frac{5}{9}
\end{align}
\end{proposition}
\begin{proof}
     The inequality \ref{imp} implies $$ (n+1 -3 h_n)^2\geq 4(n+h_n)$$ This inequality can be written as $$9 h_n^2-(6n+10) h_n+n^2-2n+1\geq 0$$ Since $0\leq h_n < \frac{n}{3}$, it follows that $$h_n\leq\frac{6n+10-\sqrt{(6n+10)^2-4 \cdot 9 (n-1)^2}}{18}$$ which simplifies to the wanted inequality.
\end{proof}

The upper bound in \eqref{bup} is more precise than the one established in Section 4, which was based on the isoperimetric inequality. This result leads to the following corollary.

\begin{corollary}\begin{align}\label{bupe}
    h_n &\leq \left \lfloor \frac{n}{3}-\frac{4}{9} \sqrt{3n+1}+\frac{5}{9}\right\rfloor
\end{align}
    
\end{corollary}

For $n \leq 7$, $h_n=0$. Therefore, for the remainder of the paper, we assume $n \geq 8$ and $a \geq 0$.

\begin{corollary} If $n=12 a^2+20 a+8$, then $h_n=4 a^2+4a+1$.
    
\end{corollary}

\begin{proof}
    By Construction \ref{qn} and  inequality \eqref{bup}, for $n=12 a^2+20 a+8$, we have $$ 4 a^2+4 a+1\leq h_n\leq \frac{3n+5-4\sqrt{3n+1}}{9}=4a^2+4 a+1.$$
\end{proof}

By a similar argument, we obtain the following result.

\begin{corollary} If $n=12 a^2+32 a+21$ then $h_n=4 a^2+8a+4$.
    
\end{corollary}

These two corollaries indicate that Construction \ref{qn}, which achieves the lower bound, provides the exact value of $h_n$ for all $n$. In the remainder of the paper, we demonstrate this result for an infinite subset of positive integers.

Let $n=12a^2+20a+8+k$, where $1\leq k\leq 24 a+31$ and $a\geq 0$. Then the inequality \eqref{bup} may be written as 
\begin{align}
    h_n & \leq (2 a+1)^2+\left \lfloor \frac{24 a+20+3k-4\sqrt{(6a+5)^2+3k}}{9}\right\rfloor
\end{align}

If $k\leq 4 a+3$, then $\sqrt{(6a+5)^2+3k}=6a+5+\alpha$, where $0<\alpha<1$, so in this case 
\begin{align}
    h_n\leq (2a+1)^2 +\left\lfloor\frac{3k-4 \alpha}{9} \right \rfloor
\end{align}

For $k=3$, it follows that $\left\lfloor\frac{3k-4 \alpha}{9} \right \rfloor=0$, so by the monotonicity of $h_n$ and Construction \ref{qn}, it follows 

\begin{proposition}
    If $0\leq k\leq 3$ and $n=12 a^2+20a+8+k$, then $h_n=4a^2+4a+1$.
\end{proposition}

For $k=4$ and $a\geq 1$ we obtain $\left\lfloor\frac{3k-4 \alpha}{9} \right \rfloor=1$ so in this case $h_n\in \{4 a^2+4a+1,4 a^2+4 a+2\}$.  

Analogously, for $5\leq k \leq 6$, $\left\lfloor\frac{3k-4 \alpha}{9} \right \rfloor=1$,  so we obtain

\begin{proposition}
    If $5\leq k\leq 6$ and $n=12 a^2+20a+8+k$ then $h_n=4a^2+4a+2$.
\end{proposition}

If $8\leq k\leq 4a+3$ and $k \not\equiv 1 \pmod 3$ we have that $\left\lfloor\frac{3k-4 \alpha}{9} \right \rfloor=\left\lfloor\frac{3k-6}{9} \right \rfloor$  implying

\begin{proposition} If $8\leq k\leq 4a+3$ and $k \not\equiv 1 \pmod 3$ and $n=12 a^2+20a+8+k$ then $h_n=(2a+1)^2 + \left\lfloor \frac{k-2}{3} \right\rfloor$.
    
\end{proposition}

If $4a+4\leq k\leq 8a+7$ then $\sqrt{(6a+5)^2+3k}=6a+6+\alpha$ where $0<\alpha<1$, so in this case $$h_n\leq (2a+1)^2+\left\lfloor \frac{3k-4-4\alpha}{9} \right\rfloor.$$

For $4a+4\leq k\leq 6a+5$ and $k\not \equiv 2 \pmod 3$ we have $\left\lfloor \frac{3k-4-4\alpha}{9} \right\rfloor=\left\lfloor \frac{3k-6}{9} \right\rfloor$. Thus,

\begin{proposition}
    For $4a+4\leq k\leq 6a+5$, $k\not \equiv 2 \pmod 3$ and $n=12 a^2+20a+8+k$ then  $$h_n=(2a+1)^2 + \left\lfloor \frac{k-2}{3} \right\rfloor$$
\end{proposition}

Applying a similar analysis, we derive the following results.

\begin{proposition}
    If $6a+6\leq k\leq 6a+7$  and $n=12 a^2+20a+8+k$ then $h_n=4a^2+6a+2$.
\end{proposition}

The monotonicity of the sequence $h_n$, together with Construction \ref{qn}, now yields a stronger result.

\begin{theorem} If $4a+4\leq k\leq 6a+7$ and $n=12 a^2+20a+8+k$ then $$h_n=(2a+1)^2 + \left\lfloor \frac{k-2}{3} \right\rfloor.$$
    
\end{theorem}

If $6a+13\leq k \leq 8a +7$, the comparison of the lower and the upper bounds reduces to checking the equality  $\left\lfloor \frac{3k-4-4\alpha}{9} \right\rfloor=\left\lfloor \frac{3k-12}{9} \right\rfloor$.  It is impossible that $\alpha \leq{1}{2}$ since in this case it would be $(6a +6+\alpha)^2\leq 36 a^2+78 a 43$ and $(6a+5)^2+3k \geq 36 a^2+7a+46$. But this is impossible. Now we can directly check that the equality holds for $k\not\equiv 0\pmod 3$. Therefore,

\begin{theorem}If $6a+13\leq k \leq 8a +7$ and $k\not\equiv 0\pmod 3$ then $$h_n=(2a+1)^2+2 +\left\lfloor \frac{k-10}{3} \right\rfloor$$
    \end{theorem}

Putting $k=6a+10$ we find $\left\lfloor \frac{3k-4-4\alpha}{9} \right\rfloor=2a+2$ so
\begin{proposition}
$h_n=4a^2+6a+3$ for $n=12 a^2+26a+18$.
\end{proposition}

\begin{theorem} For $9a+10\leq k\leq 12a +12$ and $n=12a^2+20a+8+k$ it holds $h_n=(2a+1)^2+\left\lfloor \frac{k-4}{3}\right\rfloor$.
    
\end{theorem}

\begin{proposition}If $8a+8\leq k \leq 9a+8$ and $k\not\equiv 0 \pmod 3$ then $h_n=(2a+1)^2+\left\lfloor \frac{k-4}{3}\right\rfloor$
    
\end{proposition} 

For $12 a+13\leq k\leq 16a +18$ the value of $\sqrt{(6a+5)^2+3k}=6a+8+\alpha$ where $0\leq \alpha<1$. It holds that $$h_n\leq (2a+1)^2+\left\lfloor \frac{3k-12-4\alpha}{3}\right\rfloor$$

For $k=12a+16$, we obtain $h_n\leq (2a+1)^2+4a+3$ so we can deduce that

\begin{proposition} For $12a+13\leq k\leq 12a+16$ $h_n=(2a+1)^2+4a+3$
    
\end{proposition}

In the same manner, we can establish the following.

\begin{proposition} For $12a+18\leq k\leq 12a+19$ $h_n=(2a+1)^2+4a+3$
    
\end{proposition}

By comparing $\left\lfloor \frac{3k-19}{9}\right\rfloor$ and $\left\lfloor \frac{3k-12-4\alpha}{9}\right\rfloor$, we obtain the following result.

\begin{theorem} For $12a+21\leq k\leq 16 a+8$ and $k\not \equiv 2\pmod 3$ $$h_n=(2a+1)^2+4 +\left\lfloor \frac{k-18}{3}\right\rfloor$$
    
\end{theorem}

For $16 a+19\leq k\leq 20a +24$ the value of $\sqrt{(6a+5)^2+3k}=6a+9+\alpha$ where $0<\alpha<1$. It holds that $$h_n\leq (2a+1)^2+\left\lfloor \frac{3k-12-4\alpha}{3}\right\rfloor$$

Then 
\begin{theorem} For $16a+19\leq k\leq 18 a+21$ and  $$h_n=(2a+1)^2+4 +\left\lfloor \frac{k-18}{3}\right\rfloor$$
    
\end{theorem}

\begin{proposition} For $18a+22 \leq k\leq 18 a+23$ $h_n=(2a+1)^2+6a+5$.
    
\end{proposition}

\begin{proposition} For $k=18 a+26$ $h_n=(2a+1)^2+6a+6$.
    
\end{proposition}

Similarly, we can show that. 

\begin{theorem}
    For $18a+29\leq k \leq 20a+24$ and $k\not \equiv1\pmod 3$ $$h_n=(2a+1)^2+6+\left\lfloor \frac{k-26}{3}\right\rfloor$$
\end{theorem}

\begin{theorem}
    For $20a+25\leq k \leq 24a+31$ $$h_n=(2a+1)^2+6+\left\lfloor \frac{k-26}{3}\right\rfloor$$
\end{theorem}

We now summarize the results of this section.

\begin{theorem} Let $n$ be a positive inter and let $a$ and $k$ be the unique non-negative integers such that $n=(12 a^2+20a+8)+k$ and $0\leq k\leq 24a+31$. If $n\leq 7$  then $$h_n=0.$$ If $k$ satisfies one of the following conditions 

\begin{itemize}
    \item $0\leq k\leq 3$
    \item $5\leq k\leq 6$
    \item $8\leq k\leq 4a+3$ and $k \not\equiv 1 \pmod 3$
    \item $4a+4\leq k\leq 6a+7$
    \item $k=6a+10$ 
    \item $6a+13\leq k \leq 9a+8$ and $k\not\equiv 0 \pmod 3$ 
    \item $9a+10\leq k\leq 12a+16$
    \item $12a+18\leq k\leq 12a+19$
    \item $12a+21\leq k\leq 16 a+8$ and $k\not \equiv 2\pmod 3$ 
    \item $16a+19\leq k\leq 18 a+23$
    \item $k=18 a+26$
    \item $18a+29\leq k \leq 20a+24$ and $k\not \equiv1\pmod 3$
    \item $20a+25\leq k \leq 24a+31$
\end{itemize}

then $$h_n=\left \lfloor \frac{n}{3}-\frac{4}{9} \sqrt{3n+1}+\frac{5}{9}\right\rfloor$$
    
\end{theorem}

For the case of $k$ not covered by the above theorem we have that $$h_n\in \left \{ \left \lfloor \frac{n}{3}-\frac{4}{9} \sqrt{3n+1}-\frac{4}{9}\right\rfloor, \left \lfloor \frac{n}{3}-\frac{4}{9} \sqrt{3n+1}+\frac{5}{9}\right\rfloor\right\}$$We leave as the open question finding the exact values of $h_n$ in these cases.

\section*{Acknowledgments}

The first author was supported by  Project No. H20240855 of the Ministry of Human Resources and Social Security of the People's Republic of China, and by the Ministry of Science, Innovations and Technological Development of the Republic of Serbia.

\end{document}